\documentclass[a4paper,reqno]{amsart}

\usepackage{amssymb}
\usepackage{amstext}
\usepackage{amsmath}
\usepackage{amscd}
\usepackage{amsthm}
\usepackage{amsfonts}
\usepackage{enumerate}
\usepackage{graphicx}
\usepackage{latexsym}
\usepackage{mathrsfs}
\usepackage{mathtools}
\usepackage[all]{xy}
\xyoption{all}

\usepackage{pstricks}
\usepackage{lscape}
\usepackage{comment}

\newtheorem{theorem}{Theorem}[section]

\newtheorem{proposition}[theorem]{Proposition}
\newtheorem{definition-proposition}[theorem]{Definition-Proposition}

\theoremstyle{definition}
\newtheorem{definition}[theorem]{Definition}

\newtheorem{remark}[theorem]{Remark}
\newtheorem{example}[theorem]{Example}

\newcommand{\Ext}{\operatorname{Ext}\nolimits}

\newcommand{\Hom}{\operatorname{Hom}\nolimits}
\newcommand{\End}{\operatorname{End}\nolimits}

\renewcommand{\mod}{\mathsf{mod}\hspace{.01in}}

\newcommand{\add}{\mathsf{add}\hspace{.01in}}

\newcommand{\tilt}{\mbox{\rm tilt}\hspace{.01in}}
\newcommand{\sttilt}{\mbox{\rm s$\tau$-tilt}\hspace{.01in}}

\newcommand{\gldim}{\mathop{{\rm gl.dim}}\hspace{.01in}}
\newcommand{\RHom}{\mathbf{R}\strut\kern-.2em\operatorname{Hom}\nolimits}

\numberwithin{equation}{section}

\usepackage{paralist}

\def\id{\mathop{\rm id}\nolimits}
\def\pd{\mathop{\rm pd}\nolimits}
\def\add{\mathop{\rm add}\nolimits}

\def\gldim{\mathop{\rm gldim}\nolimits}

\begin{document}
\title{Classifying tilting modules over the Auslander algebras of Radical square zero Nakayama algebras}
\thanks{2000 Mathematics Subject Classification: 16G10,16E10}
\thanks{Keywords: Auslander algebra, tilting module, support $\tau$-tilting module }
 \thanks{ The author is supported by NSFC
(Nos. 11571164 and 11671174)}

\author{Xiaojin Zhang}
\address{X. Zhang: 1. School of Mathematics and Statistics, Jiangsu Normal University, Xuzhou, 221116, P. R. China.}
\email{xjzhang@jsnu.edu.cn, xjzhangmaths@163.com}
\maketitle
\begin{abstract} Let $\Lambda$ be a radical square zero Nakayama algebra with $n$ simple modules and let $\Gamma$ be the Auslander algebra of $\Lambda$. Then every indecomposable direct summand of a tilting $\Gamma$-module is either simple or projective. Moreover, if $\Lambda$ is self-injective, then the number of tilting $\Gamma$-modules is $2^n$; otherwise, the number of tilting $\Gamma$-modules is $2^{n-1}$.
\end{abstract}

\section{Introduction}

Tilting theory has been very essential in the representation theory of finite dimensional algebras since 1970S (see \cite{B,BGP,BrB,HaR}and so on). The most important notion of tilting theory is tilting module. So it is interesting to classify tilting modules over a given algebra. For the recent development of this topics, we refer to \cite {AsSS,AnHK,CX,CrS, NRTZ, PS, IZ1}.

Auslander algebras are also important in the representation theory of finite dimensional algebras \cite{AuRS}. Tilting modules over Auslander algebras have got a lot of attention. In 2000, Br\"{u}stle and his coauthors \cite{BHRR} studied the tilting modules over the Auslander algebra of $K[x]/(x^n)$ and showed that the number of tilting modules is $n!$. Indeed, this algebra is an Auslander algebra of a Nakayama local algebra. In 2008, Kajita \cite{K} studied the number of tilting modules over the Auslander algebra of hereditary algebra of Dynkin type in his master's thesis. In 2014, Adachi, Iyama and Reiten \cite{AIR} introduced the notions of $\tau$-tilting modules as generalizations of tilting modules. For more details of $\tau$-tilting modules, we refer to \cite{DIJ,IJY,J,Mi,We,Z1,Zi1,Zi2} and the references there. In 2016, Iyama and Zhang \cite{IZ1,Ts} revisit the tilting modules over the Auslander algebra of $K[x]/(x^n)$ and built a bijection between the tilting modules over the Auslander algebra of $K[x]/(x^n)$ and the support $\tau$-tilting modules over the corresponding preprojective algebra of $A_{n-1}$ \cite {Mi}. All algebras studied in \cite{K,BHRR,IZ1} are the Auslander algebras of some Nakayama algebras. For an arbitrary Auslander algebra $\Lambda$, finding all the tilting $\Lambda$-modules seems difficult.

On the other hand, radical square zero algebras have gained a lot of attention in the representation theory of finite-dimensional algebras. Auslander, Reiten and Smal$\phi$ \cite{AuRS} showed that these algebras are stable equivalent to hereditary algebras. Chen \cite{C} showed that these algebras are either self-injective or CM-free (all finitely generated Gorenstein projective modules are projective). The Hochschild cohomology algebra of radical square zero algebras are studied by Cibils \cite{Ci}. In 2018, Wang \cite {W} studied the Tate-Hochschild cohomology of radical square zero algebras. Moreover, $\tau$-tilting modules over these algebras have studied by Adachi \cite{A2} and the author \cite{Z2}. Aoki \cite{Ao} classified torsion classes over algebras of radical square zero. It is showed by Auslander that there is a bijection between Auslander algebras and algebras of finite representation type. It is natural to consider the Auslander algebras of radical square zero algebras of finite representation type.

In this paper, we focus on tilting modules over the Auslander algebras of radical square zero Nakayama algebras and show the following theorem.
\begin{theorem} Let $\Lambda$ be a radical square zero Nakayama algebra with $n$ simple modules and let $\Gamma$ be the Auslander algebra of $\Lambda$. If $\Lambda$ is self-injective, then the number of tilting $\Gamma$-modules is $2^n$; otherwise, the number of tilting $\Gamma$-modules is $2^{n-1}$.
\end{theorem}

We show the organization of this paper as follows:
In Section 2, we recall some preliminaries and show the main result.

Throughout this paper, all algebras are finite-dimensional algebras and all modules are finitely generated right modules. For a tilting module, we mean the classical tilting module. We use $\tau$ to denote the Auslander-Reitein translation functor.

\bigskip
\noindent{\bf Acknowledgement} This note can be seen as an application of the published result proved by Iyama and the author. The author would like to thank Iyama for useful suggestions and allowing him to write this note alone.

\section{Main result and its proof}

In this section, we firstly recall some preliminaries on tilting modules, $\tau$-tilting modules and Auslander algebras. Then we prove the main result.

For an algebra $\Lambda$, denote by $\gldim\Lambda$ the global dimension of $\Lambda$ and by $I^i(\Lambda)$ the $(i+1)$-th term of minimal injective presentation of $\Lambda$. We have the following definition of Auslander algebras (resp. Auslander $1$-Gorenstein algebras).

\begin{definition}\label{2.2} An algebra $\Lambda$ is called an Auslander algebra if $\gldim\Lambda\leq2$ and $I^0(\Lambda), I^1(\Lambda)$ are projective. An algebra $\Lambda$ is called an Auslander $1$-Gorenstein if $I^0(\Lambda)$ is projective (\cite{IwS}).
\end{definition}

Recall that an algebra $R$ is of finite representation type if $\mod R=\add M$. Now we can give the definition of Auslander algebras.
Let $R$ be an algebra of finite representation type and $M$ be the additive generator of $\mod R$. We call $\Lambda=\End_R M$ the Auslander algebra of $R$.

For any $M\in\mod\Lambda$, denote by $\pd_\Lambda M$ (resp. $\id_\Lambda M$ ) the projective dimension (resp. injective dimension) of $M$. Denote by $|M|$ the number of non-isomorphic indecomposable direct summand of $M$. Now we recall the definition of tilting modules.

\begin{definition}\label{2.3} Let $\Lambda$ be an algebra and $T\in \mod\Lambda$. $T$ is called a tilting module if the following are satisfied:
\begin{itemize}
\item[\rm(1)] $\pd_\Lambda T\leq 1$

\item[\rm(2)] $\Ext_{\Lambda}^i(T,T)=0$ for $i\geq 1$.

\item[\rm(3)] $|T|=|\Lambda|$.
\end{itemize}
\end{definition}

In the following we recall the definition of $\tau$-tilting modules introduced in \cite{AIR}.

\begin{definition}\label{2.4} Let $\Lambda$  be an algebra and $M\in\mod\Lambda$. $M$ is called $\tau$-rigid if $\Hom_{\Lambda}(M,\tau M)=0$. $M$ is $\tau$-tilting if $M$ is $\tau$-rigid and $|M|=|\Lambda|$. Moreover $M$ is call a support $\tau$-tilting module if it is a $\tau$-tilting module over $\Lambda/(e)$, where $e$ is an idempotent of $\Lambda$.
\end{definition}

The following results in \cite{IZ2} shows the connection between the tilting modules over Auslander 1-Gorenstein algebras $\Lambda$ and support $\tau$-tilting modules over the factor algebra of $\Lambda$.

\begin{theorem}\label{2.5} Let $\Lambda$ be an Auslander $1$-Gorenstein algebra and $e
$ be an idempotent such that $e\Lambda$ is the additive generator of projective-injective modules. Then there is a bijection between the set
$\tilt\Lambda$ of tilting modules over $\Lambda$ and the set $\sttilt\Lambda/(e)$ of support $\tau$-tilting modules over $\Lambda/(e)$.
\end{theorem}





 Let $\Gamma$ be the Auslander algebra of a radical square zero Nakayama algebra $\Lambda$ with $n$ simple modules. Now we need to show the structure of $\Gamma$.

 \begin{proposition}\label{3.1} Let $\Gamma$ be the Auslander algebra of a radical square zero Nakayama algebra $\Lambda$ with $n$ simple modules.
 \begin{itemize}
\item[\rm(1)] If $\Lambda$ is not self-injective, then $\Gamma$ is given by the quiver $Q_1:$ $$\xymatrix{1&2\ar[l]_{a_2}&3\ar[l]_{a_3}&\cdots \ar[l]&2n-2\ar[l]_{a_{2n-2}}&2n-1\ar[l]_{a_{2n-1}}}$$ with the relations: $a_{2k-1}a_{2k-2}=0$ for $1\leq k\leq n$;
\item[\rm(2)] If $\Lambda$ is self-injective, then $\Gamma$ is given by the quiver $Q_2:$
$$
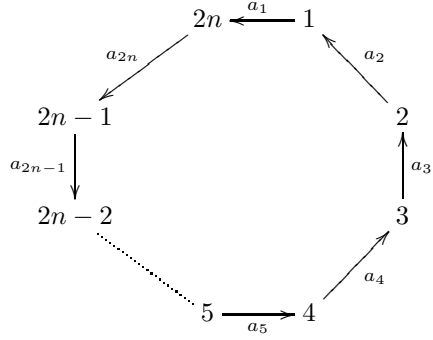
\xymatrix{
&2n\ar[ld]_{a_{2n}}&1\ar[l]_{a_1}&\\
2n-1\ar[d]_{a_{2n-1}}&&&2\ar[ul]_{a_2}\\
2n-2&&&3\ar[u]_{a_3}\\
&5\ar@{.}[lu]\ar[r]_{a_5}&4\ar[ur]_{a_4}& }$$
with the relations: $a_{2k}a_{2k-1}=0$ for $1\leq k\leq n$.
\end{itemize}
\end{proposition}

\begin{proof} By using straight calculation, one gets the quiver of $\Gamma$.
\end{proof}

 In the rest of this section, we assume that $\Lambda$ and $\Gamma$ are as in Proposition \ref{3.1}.

 For a module $T\in \mod \Lambda$, denote by $Gen(T)$ the subcategory of $\mod\Lambda$ generated by $T$. For two tilting modules $T_1$ and $T_2$, we call $T_1\leq T_2$ if $Gen(T_1)\subseteq Gen(T_2)$. Now we show the properties of tilting modules over $\Gamma$.

\begin{proposition}\label{3.2} Let $T$ be an arbitrary tilting module over $\Gamma$.
 \begin{itemize}
 \item [\rm(1)] Then any indecomposable direct summand $M$ of $T$ is either projective or a simple socle of an indecomposable projective-injective $\Gamma$-module.
 \item[\rm(2)] For any indecomposable projective and non-injective direct summand $P$ of $T$, the mutation over $P$ is given by the exact sequence $0\rightarrow P\rightarrow I^0({\rm soc}P)\rightarrow S\rightarrow0$ with $S$ simple.

 \item[\rm (3)] The minimal tilting module is a direct sum of projective-injective modules and simple modules of projective dimension $1$.
 \end{itemize}
\end{proposition}

\begin{proof} (1) By \cite[Chapter V, Theorem 3.2]{AsSS}, $\Gamma$ is a Nakayama algebra with Loewy length at most $3$.
By \cite[Lemma 3.2] {Z1}, every tilting module $T$ can be a submodule of a projective module. And hence any indecomposable direct summand $M$ should be a submodule of a projective module. Since $\Gamma$ is an Auslander algebra, $M$ can be embedded into a projective-injective module by \cite[Proposition 10]{IwS}. Since $\Gamma$ is a Nakayama algebra Loewy length at most $3$, then it is not difficult to show that $M$ is either projective or the simple socle of an indecomposable projective-injective modules.

(2) Since $T\setminus P$ is faithful, then the mutation over $P$ always exists by \cite[Proposition 1.3]{CHU}. And hence we get an exact sequence $0\rightarrow P\rightarrow T_1\rightarrow U\rightarrow 0$ with $U$ indecomposable. If $P$ is simple projective, then one gets the assertion since $\Gamma$ is a Nakayama algebra. Otherwise, the length of $P$ should be $2$ by Proposition \ref{3.1}. Then the top of $P$ is of projective dimension $2$ which is not an indecomposable direct summand of $T$. Then by (1) we get $\Hom_\Gamma(P, T')=0$ for any $T'\not=I^0({\rm soc}P)$. And therefore the minimal left $\add{(T\setminus P)}$ of $P$ should be $0\rightarrow P\rightarrow I^0({\rm soc}P)\rightarrow S\rightarrow0$.

(3) By \cite[Theorem 3.4]{IZ2}, one gets that the minimal tilting module according to the partial order defined above should be $I^0(\Gamma)\oplus \Omega^{-1}\Gamma$. By Proposition 3.1, $\Gamma$ is an Auslander algebra with $\Omega^{-1}\Gamma$ semi-simple. So the assertion holds.
\end{proof}

To show the main result, we need the following observation.

\begin{proposition}\label{3.3} Let $A$ be a semi-simple algebra with $n$ simple modules. Then the number of support $\tau$-tilting $A$-modules is $2^n$.
\end{proposition}
\begin{proof} We only need to show that any direct sum of non-isomorphic simple modules is a support $\tau$-tilting $A$-module. Let $m$ be an arbitrary integer with $1\leq m\leq n$ and $S_i$ be a simple $A$-module for $1\leq i\leq n$. Then $\oplus_{i=1}^m S_i$ is also projective since $A$ is semi-simple. And hence $\oplus_{i=1}^m S_i$ is $\tau$-rigid. Let $e$ be the idempotent $1-\Sigma_{i=1}^m e_i$ with $e_i$ the idempotent according to the simple module $S_i$. Then $\oplus_{i=1}^m S_i$ is a $\tau$-tilting (in fact tilting module) over $A/(e)$. Notice that the module $0$ is support $\tau$-tilting, then the number of support $\tau$-tilting modules over $A$ should be $\Sigma_{m=0}^{n} C_{n}^{m}=2^n$. The assertion holds.
\end{proof}

For a vertex $i$ in a quiver $Q$, we denote by $P(i)$, $I(i)$ and $S(i)$ the indecomposable projective, injective and simple module according to the $i$, respectively. Now we are ready to show the following main result.

\begin{theorem}\label{3.4}Let $\Lambda$ be a radical square zero Nakayama algebra with $n$ simple modules and let $\Gamma$ be the Auslander algebra of $\Lambda$. If $\Lambda$ is self-injective, then the number of tilting $\Gamma$-modules is $2^n$; otherwise, the number of tilting $\Gamma$-modules is $2^{n-1}$.
\end{theorem}

\begin{proof}
(1) The case of $\Lambda$ is self-injective.
By Proposition \ref{3.1}, one gets the indecomposable projective-injective $\Gamma$-modules are as follows:
$P(2n)=I(2n-1), P(2n-1)=I(2n-3), \cdots, P(3)=I(1), P(1)=I(2n-1)$. Take the idempotent $e=e_{2n}+e_{2n-1}+\cdots+e_1$. Then the factor algebra $\Gamma/(e)$ is a semi-simple algebra with $n$ simple modules. By Theorem \ref{2.5} and Proposition \ref{3.3}, we get the number of tilting modules over $\Gamma$ is $2^n$.

(2) The case of $\Lambda$ is non-self-injective.

By Proposition \ref{3.1}, one gets the indecomposable projective-injective modules are as follows:
$P(2n-1)=I(2n-2),P(2n-2)=I(2n-4),\cdots, P(4)=I(2), P(2)=I(1)$. Take the idempotent $e=e_{2n-1}+e_{2n-2}+\cdots+e_2$. Then the factor algebra $\Gamma/(e)$ is semi-simple with $n-1$ simple modules. Then by Theorem \ref{2.5} and Proposition \ref{3.3}, one gets that the number of tilting $\Gamma$-modules is $2^{n-1}$.
\end{proof}

Now we give an example to show our main results:

\begin{example} Let $\Lambda$  and $\Gamma$ be as in Proposition \ref{3.1} and let $n=3$.
\begin{itemize}
\item[\rm (1)] If $\Lambda$ is non-self-injective, then the tilting $\Gamma$-modules are as follows:

$T_1=\Gamma, T_2=P(5)\oplus P(4)\oplus S(4)\oplus P(2)\oplus P(1)$

$ T_3=P(5)\oplus P(4)\oplus P(3)\oplus P(2)\oplus S(2), T_4=P(5)\oplus P(4)\oplus S(4)\oplus P(2)\oplus S(2)$.

\item[\rm (2)] If $\Lambda$ is self-injective, then the tilting $\Gamma$-modules are as follows:

$T_1=\Gamma, T_2=S(1)\oplus P(5)\oplus P(4)\oplus P(3)\oplus P(2)\oplus P(1)$;

$T_3=P(6)\oplus P(5)\oplus S(5)\oplus P(3)\oplus P(2)\oplus P(1), T_4=P(6)\oplus P(5)\oplus P(4)\oplus P(3)\oplus S(3)\oplus P(1)$;

$T_5=S(1)\oplus P(5)\oplus S(5)\oplus P(3)\oplus P(2)\oplus P(1), T_6=P(6)\oplus P(5)\oplus S(5)\oplus P(3)\oplus S(3)\oplus P(1)$;

$T_7=S(1)\oplus P(5)\oplus P(4)\oplus P(3)\oplus S(3)\oplus P(1), T_8=S(1)\oplus P(5)\oplus S(5)\oplus P(3)\oplus S(3)\oplus P(1)$.

\end{itemize}
\end{example}

We end the paper with the following remarks.

\begin{remark}
\begin{itemize}

\item[\rm (1)] Let $\Gamma$ be the Auslander algebra of $K[x]/(x^2)$. Then it is showed in \cite{BHRR, IZ1} that the number of tilting modules over $\Gamma$ should be $2$ which is equal to $2^1$. So Theorem \ref{3.4} is also true.
\item [\rm (2)]Tilting modules of finite projective dimension over the the Auslander algebra of $K[x]/(x^n)$ have been classified by Geuenich \cite{Je}. However, the classification of tilting modules of finite projective dimension over the Auslander algebras of radical square zero Nakayama algebras is still open.
\item [\rm (3)] Adachi in \cite{A1} classified $\tau$-tilting modules over Nakayama algebras in terms of triangulations. It is also interesting to find the correspondence between tilting modules and triangulations.

\end{itemize}
\end{remark}

\end{document}